\documentclass{article}

\usepackage[a4paper, top=3cm, left=3cm, right=2cm, bottom=2cm]{geometry}
\usepackage[dvips]{graphicx}
\usepackage[active]{srcltx}
\usepackage{amsmath, amsthm, pb-diagram}
\usepackage{amssymb}
\usepackage{amsfonts}
\usepackage{enumerate, fancyhdr, dsfont}

    \newcommand{\C}{\mathbb{C}}

    \newcommand{\lcom}{\textquotedblleft}
    \newcommand{\rcom}{\textquotedblright}

    \newcommand{\modulo}{\mbox{\rm\ mod\ }}

    \newcommand{\cj}[2]{ \left\{ {#1} \ / \ {#2} \right\} }

    \newcommand{\thzfc}{\mathrm{ZFC}}

    \newcommand{\Bwf}{\mathcal{B}}

    \newcommand{\Fwf}{\mathcal{F}}

    \newcommand{\Iwf}{\mathcal{I}}
    
    \newcommand{\Mwf}{\mathcal{M}}
    \newcommand{\Nwf}{\mathcal{N}}

    \newcommand{\bfrak}{\mathfrak{b}}
    \newcommand{\cfrak}{\mathfrak{c}}
    \newcommand{\dfrak}{\mathfrak{d}}

    \newcommand{\vacio}{\varnothing}

    \newcommand{\frestr}{\!\!\upharpoonright\!\!}

    \newcommand{\add}{\mbox{\rm add}}
    \newcommand{\cov}{\mbox{\rm cov}}
    \newcommand{\non}{\mbox{\rm non}}
    \newcommand{\cof}{\mbox{\rm cof}}
    \newcommand{\Sl}{\mbox{\rm Sl}}

    \newcommand{\Aor}{\mathds{A}}
    \newcommand{\Bor}{\mathds{B}}
    \newcommand{\Cor}{\mathds{C}}
    \newcommand{\Dor}{\mathds{D}}
    \newcommand{\Eor}{\mathds{E}}
    \newcommand{\Por}{\mathds{P}}
    \newcommand{\Qor}{\mathds{Q}}
    
    \newcommand{\Sor}{\mathds{S}}
    
    \newcommand{\Qnm}{\dot{\mathds{Q}}}
    \newcommand{\Rnm}{\dot{\mathds{R}}}
    \newcommand{\Snm}{\dot{\mathds{S}}}
    \newcommand{\Tnm}{\dot{\mathds{T}}}
    \newcommand{\Anm}{\dot{\mathds{A}}}
    \newcommand{\Bnm}{\dot{\mathds{B}}}
    \newcommand{\Cnm}{\dot{\mathds{C}}}
    \newcommand{\Dnm}{\dot{\mathds{D}}}
    \newcommand{\Enm}{\dot{\mathds{E}}}

    \newcommand{\cf}{\mbox{\rm cf}}

    \newcommand{\conj}{{\ \mbox{\scriptsize $\wedge$} \ }}

    \newcommand{\imp}{{\ \mbox{$\Rightarrow$} \ }}

    \newcommand{\sii}{{\ \mbox{$\Leftrightarrow$} \ }}

\title{Matrix iterations and Cichon's diagram}
\author{Diego Alejandro Mej\'ia\thanks{Supported by the Monbukagakusho (Ministry of Education, Culture, Sports, Science and Technology) Scholarship, Japan.}\thanks{GraduateSchool of System Informatics, Kobe University, Kobe, Japan.}\thanks{damejiag@kurt.scitec.kobe-u.ac.jp}
}

\date{April 19th, 2012}
\renewcommand\footnotemark{}

\begin{document}

\makeatletter
\def\@roman#1{\romannumeral #1}
\makeatother

\theoremstyle{plain}
  \newtheorem{theorem}{Theorem}
  \newtheorem{corollary}{Corollary}
  \newtheorem{lemma}{Lemma}
  \newtheorem{prop}{Proposition}
  \newtheorem{clm}{Claim}
  \newtheorem{exer}{Exercise}
\theoremstyle{definition}
  \newtheorem{definition}{Definition}
  \newtheorem{example}{Example}
  \newtheorem{remark}{Remark}
  \newtheorem{context}{Context}
  \newtheorem*{acknowledgements}{Acknowledgements}

\maketitle

\begin{abstract}
Using matrix iterations of ccc posets, we prove the consistency with $\thzfc$ of some cases where the cardinals on the right hand side of
Cichon's diagram take two or three arbitrary values (two regular values, the third one with uncountable cofinality). Also, mixing this with the
techniques in \cite{brendle}, we can prove that it is consistent with $\thzfc$ to assign, at the same time, several arbitrary regular values on
the left hand side of Cichon's diagram.
\end{abstract}

\section{Introduction}\label{intro}

In this work we are interested in obtaining models where \emph{the continuum is large} (we mean with this that the size of the continuum is
$\geq\aleph_3$) and where cardinal invariants in Cichon's diagram can take arbitrary regular values. So far, from \cite{brendle}
models are known where those cardinals take as values two previously fixed arbitrary regular cardinals.\\
Concerning the possibility of models where the invariants in Cichon's diagram assume three or more different values, the iteration techniques in
\cite{brendle} bring models where cardinals on the left hand side of Cichon's diagram take several arbitrary values. Nevertheless, models where
invariants of the right hand side of Cichon's diagram can assume more than two arbitrary values seem more difficult to get, furthermore, more
sophisticated techniques than the usual finite support iteration of ccc posets seem to be needed to construct such
models.\\
We use the technique of matrix iterations of ccc posets (see \cite{blsh} and \cite{BF}) to construct models of $\thzfc$ of some cases where
cardinals on the right hand side of Cichon's diagram can assume two or three arbitrary values, the greatest of them with uncountable cofinality
and the others regular. Even more, we use some of the reasonings in \cite{brendle} with this technique in order to, at the
same time, assign several arbitrary regular values to the invariants of the left hand side of the diagram.\\
Throughout this text, we refer to a member of $\omega^\omega$ (the set of functions from $\omega$ to $\omega$) or to a member of the cantor
space $2^\omega$ (the set of functions from $\omega$ to $2=\{0,1\}$) as a \emph{real}. $\Mwf$ denotes the $\sigma$-ideal of meager sets of reals
and $\Nwf$ is the $\sigma$-ideal of null sets of reals (from the context, it is possible to guess whether the reals correspond to
$\omega^\omega$ or to $2^\omega$). For $\Iwf$ being $\Mwf$ or $\Nwf$, the following cardinal invariants are defined
\begin{description}
   \item[$\add(\Iwf)$] the least size of a family $\Fwf\subseteq\Iwf$ whose union is not in $\Iwf$,
   \item[$\cov(\Iwf)$] the least size of a family $\Fwf\subseteq\Iwf$ whose union covers all the reals,
   \item[$\non(\Iwf)$] the least size of a set of reals not in $\Iwf$, and
   \item[$\cof(\Iwf)$] the least size of a cofinal subfamily of $\langle\Iwf,\subseteq\rangle$.
\end{description}
The value of each of these invariants doesn't depend on the space of reals used to define it.\\
We consider $\cfrak=2^{\aleph_0}$ (the size of the continuum) and the invariants $\bfrak$ and $\dfrak$ as given in Subsection \ref{SubsecUnbd}.
Thus, we have Cichon's diagram as in figure \ref{fig:1}.
\begin{figure}
\begin{center}
  \includegraphics[scale=0.52]{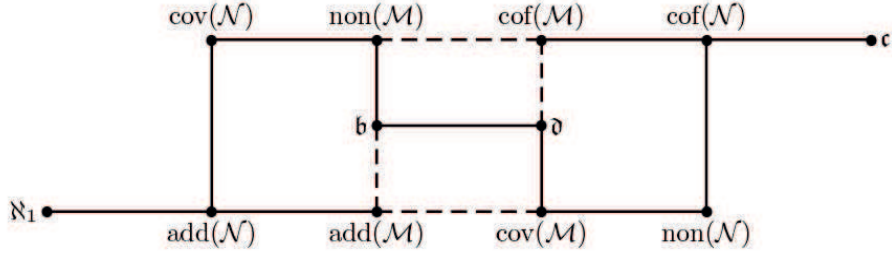}
\caption{Cichon's diagram}
\label{fig:1}
\end{center}
\end{figure}
In figure \ref{fig:1}, horizontal lines from left to right and vertical lines from down to up represent $\leq$. The dotted lines represent
$\add(\Mwf)=\min\{\bfrak,\cov(\Mwf)\}$ and $\cof(\Mwf)=\max\{\dfrak,\non(\Mwf)\}$. For basic definitions, notation and proofs regarding Cichon's
diagram, see \cite[Chapter 2]{barju} and \cite{bart}.\\
Our notation is quite standard. $\Aor$ represents the amoeba algebra, $\Bor$ the random algebra, $\Cor$ the Cohen poset, $\Dor$ is Hechler
forcing, $\Eor$ is the eventually different reals forcing and $\mathds{1}$ denotes the trivial poset $\{0\}$. Those posets are Suslin ccc
forcing notions. See \cite[Chapter 3, Section 7.4B]{barju} for definitions and properties. We abbreviate the expression \lcom finite support
iteration\rcom\ by fsi. Basic notation and knowledge about forcing can be found in \cite{kunen} and \cite{jech}.\\
This paper is structured as follows. In section \ref{SecPresUnbd}, we present preservation results in a very general setting as they are given
in \cite[Section 6.4]{barju} with some small variations of the definitions and results. At the end, some particular cases of those properties
are mentioned, previously presented in \cite{brendle}, \cite{jushe} and in \cite[Section 6.5]{barju}. The contents of this section are
fundamental results to preserve lower and upper bounds of some cardinal invariants under forcing extensions and they are used to calculate
the values of the invariants involved in the models constructed in sections \ref{SecModelLeft} and \ref{SecModelRight}.\\
Section \ref{SecModelLeft} contains extensions of some models presented in \cite{brendle} where one cardinal invariant of the right hand side of
Cichon's diagram is preserved to be large and where some invariants of the left hand side can take arbitrary regular uncountable values. The
same technique as the one presented in \cite{brendle} works to obtain such models. Some of them are used to start the constructions of the
models in section
\ref{SecModelRight}.\\
In the same general context presented in section \ref{SecPresUnbd}, preservation results about unbounded reals are contained in section
\ref{SecPresUnbdg}. Those results have been already presented in \cite{blsh} and \cite{BF} with a particular notation, but we add
Theorem \ref{suslinStar} concerning these preservation results with the properties stated in section \ref{SecPresUnbd}.\\
In section \ref{SecMatrixit}, we define the specific case of matrix iterations of ccc posets and present Corollary \ref{MainMatrixConsq} that
allows us to calculate, in a generic extension, the size of one invariant of the right hand side of Cichon's diagram. This specific case
consists of using subalgebras of Suslin ccc forcing notions as it is done in \cite{BF}, but here we consider also the case when a column of the
matrix is extended using fully some Suslin ccc forcing notion and some relations with the properties presented in section \ref{SecPresUnbd}
(Theorem \ref{PresvUnbdg}). This allows us, in section \ref{SecModelRight}, to obtain models where cardinal invariants of the right hand side of
Cichon's diagram can take two or three arbitrary values, which are the main results of this paper. There, we also use the techniques for the
models in section \ref{SecModelLeft} to assign
arbitrary regular values to some invariants of the left hand side of Cichon's diagram.\\
At the end, in section \ref{SecQ} are included some questions regarding the material of this paper.

\section{Preservation of $\sqsubset$-unbounded families}\label{SecPresUnbd}

This section contains some of the notation and results in \cite{gold}, \cite{brendle} and in \cite[Sections 6.4 and 6.5]{barju}. Throughout this
section, we fix $\kappa$ an uncountable regular cardinal and $\lambda\geq\kappa$ infinite cardinal.
\begin{context}[\cite{gold}, {\cite[Section 6.4]{barju}}]\label{ContextUnbd}
 We fix an increasing sequence $\langle\sqsubset_n\rangle_{n<\omega}$ of 2-place relations in $\omega^\omega$ such that
\begin{itemize}
   \item each $\sqsubset_n$ ($n<\omega)$ is a closed relation (in the arithmetical sense) and
   \item for all $n<\omega$ and $g\in\omega^\omega$, $(\sqsubset_n)^g=\cj{f\in\omega^\omega}{f\sqsubset_n g}$ is (closed) n.w.d.
\end{itemize}
Put $\sqsubset=\bigcup_{n<\omega}\sqsubset_n$. Therefore, for every $g\in\omega^\omega$, $(\sqsubset)^g$ is an $F_\sigma$ meager set.\\
$F\subseteq\omega^\omega$ is a \emph{$\sqsubset$-unbounded family} if, for every $g\in\omega^\omega$, there exists an $f\in F$ such that
$f\not\sqsubset g$. We define the cardinal $\bfrak_\sqsubset$ as the least size of a $\sqsubset$-unbounded family. Besides,
$D\subseteq\omega^\omega$ is a \emph{$\sqsubset$-dominating family} if, for every $x\in\omega^\omega$, there exists an $f\in D$ such that
$x\sqsubset f$. Likewise, we define the cardinal $\dfrak_\sqsubset$ as the least size of a $\sqsubset$-dominating family.\\
Given a set $Y$, we say that a real $f\in\omega^\omega$ is \emph{$\sqsubset$-unbounded over $Y$} if $f\not\sqsubset g$ for every $g\in
Y\cap\omega^\omega$.
\end{context}
Although we define Context \ref{ContextUnbd} for $\omega^\omega$, in general we can use the same notion by changing the space for the domain or
the range of $\sqsubset$ to another uncountable Polish space, like $2^\omega$ or other spaces whose members can be coded by reals in
$\omega^\omega$. This will be the case for the particular cases in the subsections \ref{SubsecCovNull} and \ref{SubsecAddNull}.
\begin{definition}\label{DefTriangle}
  For a set $F\subseteq\omega^\omega$, the property
  $(\blacktriangle,\sqsubset,F,\kappa)$ holds if, for all $X\subseteq\omega^\omega$ such that $|X|<\kappa$,
             there exists an $f\in F$ which is $\sqsubset$-unbounded over $X$.
\end{definition}
This property implies directly that $F$ is a $\sqsubset$-unbounded family and that no set of size $<\kappa$ is $\sqsubset$-dominating, that is,
\begin{lemma}\label{TriangleImpl}
   $(\blacktriangle,\sqsubset,F,\kappa)$ implies that $\bfrak_\sqsubset\leq|F|$ and $\kappa\leq\dfrak_{\sqsubset}$.
\end{lemma}
\begin{definition}[{\cite[Def. 6.4.4.3]{barju}}]\label{DefPlusProp}
  For a forcing notion $\Por$, the property
  $(+^\kappa_{\Por,\sqsubset})$ holds if, for every $\Por$-name $\dot{h}$ of a real in $\omega^\omega$, there exists a set
        $Y\subseteq\omega^\omega$ such that $|Y|<\kappa$ and, for every $f\in\omega^\omega$, if $f$ is $\sqsubset$-unbounded over $Y$,
        then $\Vdash f\not\sqsubset\dot{h}$.
\end{definition}
When $\kappa=\aleph_1$, we just write $(+_{\Por,\sqsubset})$. It is clear that
$(+^\kappa_{\Por,\sqsubset})$ implies $(+^\lambda_{\Por,\sqsubset})$.\\
The property $(+^\kappa_{\cdot,\sqsubset})$ corresponds, for some particular cases of $\sqsubset$, to the notions of $\kappa$-good,
$\kappa$-nice and $\kappa$-full discussed in \cite{brendle}.
\begin{definition}[Judah and Shelah, \cite{jushe}, \cite{brendle}, {\cite[Def 6.4.4.4]{barju}}]\label{DefGood}
   A forcing notion $\Por$ is \emph{$\kappa$-$\sqsubset$-good} if, for an arbitrary large $\chi$, whenever $M\prec H_\chi$ is such that $\Por\in
   M$ and $|M|<\kappa$, then there is an $N$ with $M\preceq N\prec H_\chi$, $|N|<\kappa$, such that
   \begin{enumerate}[(i)]
      \item $\Vdash\textrm{\lcom}\dot{G}\textrm{ is }\Por\textrm{-generic over }N$\rcom, and
      \item for all $f\in\omega^\omega$ $\sqsubset$-unbounded over $N$, $\Vdash\textrm{\lcom}f\textrm{ is }\sqsubset\textrm{-unbounded over
      }N[\dot{G}]$\rcom.
   \end{enumerate}
\end{definition}
\begin{lemma}[{\cite[Lemma 6.4.5]{barju}}]\label{EqvNice}
   If $\Por$ is $\kappa$-$\sqsubset$-good, then $(+^\kappa_{\Por,\sqsubset})$ holds. The converse is true when $\Por$ is $\kappa$-cc.
\end{lemma}
\begin{proof}
   Assume that $\Por$ is $\kappa$-$\sqsubset$-good and let $\dot{h}$ be a $\Por$-name for a real in $\omega^\omega$.
   Choose $M\prec H_\chi$, $|M|<\aleph_1$ such that $\Por,\dot{h}\in M$. Then,
   there exists an $N$ witnessing $\kappa$-$\sqsubset$-goodness for $M$, so $Y:=\omega^\omega\cap N$ witnesses $(+^\kappa_{\Por,\sqsubset})$ for
   $\dot{h}$.\\
   For the converse, let $M\prec H_\chi$ such that $|M|<\kappa$ and $\Por\in M$. By recursion, construct a sequence
   $\langle M_n\rangle_{n<\omega}$ such that, for every $n<\omega$,
   \begin{itemize}
      \item $M_0=M$,
      \item $M_n\preceq M_{n+1}\prec H_\chi$, $|M_n|<\kappa$,
      \item $\sup(M_n\cap\kappa)\subseteq M_{n+1}$ and
      \item for every $\dot{h}\in M_n$ $\Por$-name for a real in $\omega^\omega$ there exists a $Y_{\dot{h}}$ witness
            of $(+^\kappa_{\Por,\sqsubset})$ for $\dot{h}$ such that $Y_{\dot{h}}\subseteq M_{n+1}$.
   \end{itemize}
   Put $N=\bigcup_{n<\omega}M_n$. As $N\cap\kappa\subseteq N$ and $\Por$ is $\kappa$-cc, it follows that $N$ is as desired for
   $\kappa$-$\sqsubset$-goodness.
\end{proof}
The property $(+^\kappa_{\cdot,\sqsubset})$ is important for a forcing notion to preserve $\bfrak_\sqsubset$ small and $\dfrak_\sqsubset$ large.
In general,
\begin{lemma}[{\cite[Lemma 6.4.8]{barju}}]\label{PreserTriangle}
   Assume $(+^\kappa_{\Por,\sqsubset})$. Then, the statements $(\blacktriangle,\sqsubset,F,\kappa)$ and
   \lcom$\dfrak_\sqsubset\geq\lambda$\rcom\ are preserved in generic extensions of $\Por$.
\end{lemma}
\begin{proof}
   Let $\mu<\kappa$ and $\langle\dot{x}_\alpha\rangle_{\alpha<\mu}$ a sequence of $\Por$-names of reals in $\omega^\omega$. For each
   $\alpha<\mu$, there exists a $Y_\alpha\subseteq\omega^\omega$ witnessing $(+^\kappa_{\Por,\sqsubset})$ for $\dot{x}_\alpha$. Put
   $Y=\bigcup_{\alpha<\mu}Y_\alpha$. As $|Y|<\kappa$, $(\blacktriangle,\sqsubset,F,\kappa)$ implies that there is an $f\in F$ which is
   $\sqsubset$-unbounded over $Y$, so it follows that $\Vdash f\not\sqsubset\dot{x}_\alpha$ for each $\alpha<\mu$. This proves
   $\Vdash(\blacktriangle,\sqsubset,F,\kappa)$.\\
   A similar argument proves that, whenever $\dfrak_\sqsubset\geq\lambda$, $\Por$ forces that no family of size $<\lambda$ is $\sqsubset$-dominating.
\end{proof}
Lemma 2 and \cite[Thm. 6.4.12.2]{barju} gives the following result about fsi of $\kappa$-cc forcing notions.
\begin{theorem}[Judah and Shelah, \cite{jushe}, {\cite[Thm. 6.4.12.2]{barju}}, \cite{brendle}]\label{preservPlus}
   Let $\kappa$ be an uncountable cardinal, $\Por_\delta=\langle\Por_\alpha,\Qnm_\alpha\rangle_{\alpha<\delta}$ a fsi of $\kappa$-cc forcing.
   If $\forall_{\alpha<\delta}\big(\Vdash_{\Por_\alpha}(+^\kappa_{\Qnm_\alpha,\sqsubset})\big)$, then $(+^\kappa_{\Por_\delta,\sqsubset})$.
\end{theorem}
\begin{lemma}[{\cite[Thm. 6.4.7]{barju}}]\label{smallPlus}
   If $\Por$ is a poset and $|\Por|<\kappa$, then $(+^\kappa_{\Por,\sqsubset})$. In particular, $(+_{\C,\sqsubset})$ always holds.
\end{lemma}
The foregoing result, in the case of ccc posets, follows directly from Lemma \ref{EqvNice} and \cite[Thm. 6.4.7]{barju}. By a generalization of the technique of the proof of \cite[Lemma. 3.3.19]{barju}, the ccc assumption can be omitted.
\begin{proof}
   Put $\Por=\cj{p_\alpha}{\alpha<\mu}$ where $\mu:=|\Por|<\kappa$. Let $\dot{h}$ be a $\Por$-name for a real in $\omega^\omega$.
   For each $\alpha<\mu$, choose $\langle q_n^\alpha\rangle_{n<\omega}$ a decreasing sequence in $\Por$ and $h_\alpha\in\omega^\omega$ such that
   $q_0^\alpha=p_\alpha$ and, for every $n<\omega$, $q_n^\alpha\Vdash\dot{h}\frestr n=h_\alpha\frestr n$. It suffices to prove
   that, if $f\in\omega^\omega$ and $\forall_{\alpha<\mu}(f\not\sqsubset h_\alpha)$ then $\Vdash f\not\sqsubset\dot{h}$, that is,
   $\forall_{p\in\Por}\forall_{m<\omega}\exists_{q\leq p}(q\Vdash f\not\sqsubset_m\dot{h})$. Fix $p\in\Por$ and $m<\omega$, so there exists
   an $\alpha<\mu$ such that $p=p_\alpha$. As $f\not\sqsubset h_\alpha$ and $(\sqsubset_m)_f:=\cj{g\in\omega^\omega}{f\sqsubset_m g}$ is closed, there exists $n<\omega$ such that
   $[h_\alpha\frestr n]\cap (\sqsubset_m)_f=\varnothing$, so $q_n^\alpha\Vdash[\dot{h}\frestr n]\cap (\sqsubset_m)_f=[h_\alpha\frestr n]\cap (\sqsubset_m)_f=\varnothing$.
   Therefore, $q^\alpha_n\Vdash\dot{h}\notin (\sqsubset_m)_f$, that is, $q^\alpha_n\Vdash f\not\sqsubset_m\dot{h}$ with $q_n^\alpha\leq p_\alpha=p$.
\end{proof}
The following particular cases of $\sqsubset$ are presented in \cite{brendle}, \cite{jushe} and in \cite[Section 6.5]{barju}.

\subsection{Preserving non-meager sets}\label{SubsecNonMeag}

For $f,g\in\omega^\omega$, define $f\eqcirc_n g\sii\forall_{k\geq n}(f(k)\neq g(k))$, so $f\eqcirc g\sii\forall^\infty_{k\in\omega}(f(k)\neq
g(k))$. From the characterization of covering and uniformity of category (see \cite[Thm. 2.4.1 and 2.4.7]{barju}), it follows that
$\bfrak_{\eqcirc}=\non(\Mwf)$ and $\dfrak_{\eqcirc}=\cov(\Mwf)$.

\subsection{Preserving unbounded families}\label{SubsecUnbd}

For $f,g\in\omega^\omega$, define $f<^*_n g\sii\forall_{k\geq n}(f(k)<g(k))$, so $f<^*g\sii\forall^\infty_{k\in\omega}(f(k)<g(k))$. Clearly,
$\bfrak_{<^*}=\bfrak$ and $\dfrak_{<^*}=\dfrak$. $(+_{\Bor,<^*})$ holds because $\Bor$ is $\omega^\omega$-bounding, also
\begin{lemma}[Miller, \cite{miller}]\label{EvDiffPlus}
    $(+_{\Eor,<^*})$ holds.
\end{lemma}

\subsection{Preserving null-covering families}\label{SubsecCovNull}

This particular case is a variation of the case for fullness considered in section 3.2 of \cite{brendle}. Fix, from now on, $\langle
I_n\rangle_{n<\omega}$ an interval partition of $\omega$ such that $\forall_{n<\omega}(|I_n|=2^{n+1})$. For $f,g\in2^\omega$ define
$f\pitchfork_ng\sii\forall_{k\geq n}(f\frestr I_k\neq g\frestr I_k)$, so $f\pitchfork g\sii \forall^\infty_{k<\omega}(f\frestr I_k\neq g\frestr
I_k)$. Clearly, $(\pitchfork)^g$ is a co-null $F_\sigma$ meager set.
\begin{lemma}[{\cite[Lemma $1^*$]{brendle}}]\label{centeredFork}
   Given $\mu<\kappa$ infinite cardinal, every $\mu$-centered forcing notion satisfies $(+^\kappa_{\cdot,\pitchfork})$.
\end{lemma}
The proof of this result follows from the ideas of the proofs of Lemmas 1 and 6 in \cite{brendle}.\\
\begin{lemma}\label{InvforPitchfork}
   $\cov(\Nwf)\leq\bfrak_\pitchfork\leq\non(\Mwf)$ and $\cov(\Mwf)\leq\dfrak_\pitchfork\leq\non(\Nwf)$.
\end{lemma}
\begin{proof}
   $\bfrak_\pitchfork\leq\non(\Mwf)$ follows from the fact that any subset of $2^\omega$ that is $\pitchfork$-bounded should be meager.
   $\cov(\Mwf)\leq\dfrak_\pitchfork$ follows from the fact that, for any model $M$ of enough $\thzfc$ of size $<\cov(\Mwf)$,
   there exists a Cohen real $c$ over $M$, so $c$ will be $\pitchfork$-unbounded over $M$. Noting that $(\pitchfork)^g$ is a co-null $F_\sigma$
   set for every $g\in2^\omega$, we get directly that $\cov(\Nwf)\leq\bfrak_\pitchfork$ and $\dfrak_\pitchfork\leq\non(\Nwf)$.
\end{proof}

\subsection{Preserving union of null sets is not null}\label{SubsecAddNull}

Define
\[\Sl=\cj{\varphi:\omega\to[\omega]^{<\omega}}{\exists_{k<\omega}\forall_{n<\omega}(|\varphi(n)|\leq(n+1)^k)}\]
the \emph{space of slaloms}. As a Polish space, this is coded by reals in $\omega^\omega$. For $f\in\omega^\omega$ and a slalom $\varphi$,
define $f\subseteq^*_n\varphi\sii\forall_{k\geq n}(f(k)\in\varphi(k))$, so $f\subseteq^*\varphi\sii \forall_{k<\omega}^\infty(f(k)\in
\varphi(k))$. From the characterization given by \cite[Thm. 2.3.9]{barju}, $\bfrak_{\subseteq^*}=\add(\Nwf)$ and
$\dfrak_{\subseteq^*}=\cof(\Nwf)$.
\begin{lemma}[Judah and Shelah, \cite{jushe} and \cite{brendle}]\label{centeredAddNull}
   Given $\mu<\kappa$ infinite cardinals, every $\mu$-centered forcing notion satisfies $(+^\kappa_{\cdot,\subseteq^*})$.
\end{lemma}
\begin{lemma}[Kamburelis, \cite{kamburelis}]\label{spamPlus}
   Every boolean algebra with a strictly positive finitely additive measure (see \cite{kamburelis} for this concept)
   satisfies $(+_{\cdot,\subseteq^*})$. In particular, subalgebras of the random
   algebra satisfy that property.
\end{lemma}

\section{Models for the left hand side of Cichon's diagram}\label{SecModelLeft}

For any infinite cardinal $\lambda$, we use the notation
\begin{description}
  \item[$\mathbf{GCH}_\lambda$] For any infinite cardinal $\mu$,
       \[2^\mu=\left\{
            \begin{array}{ll}
               \lambda & \textrm{if $\mu<\cf(\lambda)$,}\\
               \lambda^+ & \textrm{if $\cf(\lambda)\leq\mu<\lambda$,}\\
               \mu^+ & \textrm{if $\lambda\leq\mu$.}
            \end{array}\right.\]
\end{description}
Throughout this section, we fix $\mu_1\leq\mu_2\leq\mu_3\leq\kappa$ regular uncountable cardinals, and $\lambda\geq\kappa$ cardinal. Also, fix
$V$ a model of $\thzfc+\mathbf{GCH}$. Models obtained in this section are direct consequences of the techniques used to obtain models in
\cite{brendle}.\\
The following result will be the starting point for all the models we get in this section and in section \ref{SecModelRight}.
\begin{theorem}\label{LeftCichonCovMLarge}
   In $V$, assume $\cf(\lambda)\geq\mu_3$. Then, there exists a ccc poset that forces $\mathbf{GCH}_\lambda$,
   $\add(\Nwf)=\mu_1$, $\cov(\Nwf)=\mu_2$, $\bfrak=\non(\Mwf)=\mu_3$ and $\cov(\Mwf)=\cfrak=\lambda$.
\end{theorem}
\begin{proof}
   Any generic extension $V^1$ of the poset resulting from the $\mu_1$-stage fsi of $\Aor$ satisfies
   $\add(\Nwf)=\cfrak=\mu_1$ and $\mathbf{GCH}_{\mu_1}$, also, $A:=\omega^\omega\cap V^1$ has size $\mu_1$
   and $(\blacktriangle,\subseteq^*,A,\mu_1)$ holds.\\
   In $V^1$ perform a fsi $\langle\Por^1_\alpha,\Qnm^1_\alpha\rangle_{\alpha<\mu_2}$ such that
   \begin{itemize}
      \item for $\alpha$ even, $\Qnm^1_\alpha=\Bnm$ ($\Por^1_\alpha$-name for $\Bor$) and
      \item for $\alpha$ odd $\Qnm^1_\alpha$ is a $\Por^1_\alpha$-name for a subalgebra of $\Aor$ of size $<\mu_1$.
   \end{itemize}
   By a book-keeping process, we make sure that all the subalgebras of $\Aor$ (corresponding to the final stage of the iteration)
   of size $<\mu_1$ appear in some stage of the iteration. Lemmas \ref{smallPlus}, \ref{spamPlus} and Theorem \ref{preservPlus} yield that
   all the stages of the iteration satisfy $(+^{\mu_1}_{\cdot,\subseteq^*})$. Now, if $V^2$ is a generic extension of this iteration,
   $(\blacktriangle,\subseteq^*,A,\mu_1)$ is preserved, so $\add(\Nwf)\leq\mu_1$ by Lemma \ref{TriangleImpl}. For $\add(\Nwf)\geq\mu_1$ it is enough to
   note that, from any sequence of $<\mu_1$ Borel null sets coded in $V^2$, a subalgebra of $\Aor$ of size $<\mu_1$ can be defined in such a way that
   a generic extension of it adds a Borel null set that covers those $<\mu_1$ Borel sets. From the cofinally many random reals added in the
   iteration, $\cov(\Nwf)=\cfrak=\mu_2$ holds. $\mathbf{GCH}_{\mu_2}$ is easy to get. Put $B:=2^\omega\cap V^2$, so $(\blacktriangle,\pitchfork,
   B,\mu_2)$ holds and $|B|=\mu_2$.\\
   Now, in $V^2$, perform a fsi $\langle\Por^2_\alpha,\Qnm^2_\alpha\rangle_{\alpha<\mu_3}$ such that
   \begin{itemize}
     \item for $\alpha\equiv0\modulo{3}$, $\Qnm^2_\alpha=\Dnm$ ($\Por^2_\alpha$-name for $\Dor$),
     \item for $\alpha\equiv1\modulo{3}$, $\Qnm^2_\alpha$ is a $\Por^2_\alpha$-name for a subalgebra of $\Aor$ of size $<\mu_1$ and
     \item for $\alpha\equiv2\modulo{3}$, $\Qnm^2_\alpha$ is a $\Por^2_\alpha$-name for a subalgebra of $\Bor$ of size $<\mu_2$.
   \end{itemize}
   Like in the previous step, we make sure to use all the subalgebras of $\Aor$ of size $<\mu_1$ and all subalgebras of $\Bor$ of size $<\mu_2$
   in the iteration. As $\Dor$ is $\sigma$-centered, all stages of this iteration satisfy $(+^{\mu_1}_{\cdot,\subseteq^*})$ and
   $(+^{\mu_2}_{\cdot,\pitchfork})$. If $V^3$ is any generic extension of this iteration, $(\blacktriangle,\subseteq^*,A,\mu_1)$
   and $(\blacktriangle,\pitchfork,
   B,\mu_2)$ are preserved, so $\add(\Nwf)\leq\mu_1$ and $\cov(\Nwf)\leq\mu_2$. The same argument as before yields $\add(\Nwf)=\mu_1$ and a
   similar argument can be used to get $\cov(\Nwf)\geq\mu_2$ (from less than $\mu_2$ Borel null sets coded in $V^3$ we can get a subalgebra of
   $\Bor$ of size $<\mu_2$ that adds a real that is not in any of those sets). The $\mu_3$ Hechler reals give us $\bfrak\geq\mu_3$. $\cfrak=\mu_3$
   and $\mathbf{GCH}_{\mu_3}$ are easy to prove. Note that $(\blacktriangle,\eqcirc,C,\mu_3)$ holds for $C:=\omega^\omega\cap V^3$ and $|C|=\mu_3$.\\
   Finally, in $V^3$, perform a fsi $\langle\Por^3_\alpha,\Qnm^3_\alpha\rangle_{\alpha<\lambda}$ such that
   \begin{itemize}
     \item for $\alpha\equiv0\modulo3$, $\Qnm^3_\alpha$ is a $\Por^3_\alpha$-name for a subalgebra of $\Aor$ of size $<\mu_1$,
     \item for $\alpha\equiv1\modulo3$, $\Qnm^3_\alpha$ is a $\Por^3_\alpha$-name for a subalgebra of $\Bor$ of size $<\mu_2$ and
     \item for $\alpha\equiv2\modulo3$, $\Qnm^3_\alpha$ is a $\Por^3_\alpha$-name for a subalgebra of $\Dor$ of size $<\mu_3$.
   \end{itemize}
   We make sure to use all such subalgebras. Every stage of the iteration satisfy $(+^{\mu_1}_{\cdot,\subseteq^*})$,
   $(+^{\mu_2}_{\cdot,\pitchfork})$ and $(+^{\mu_3}_{\cdot,\eqcirc})$. If $V^4$ is a generic extension of this iteration, as $(\blacktriangle,\subseteq^*,A,\mu_1)$, $(\blacktriangle,\pitchfork,B,\mu_2)$ and $(\blacktriangle,\eqcirc,C,\mu_3)$ are preserved, by a similar argument
   as before, we get $\add(\Nwf)=\mu_1$, $\cov(\Nwf)=\mu_2$ and $\mu_3\leq\bfrak\leq\non(\Mwf)\leq\mu_3$. $\cfrak=\lambda$ and
   $\mathbf{GCH}_\lambda$. It remains to prove $\cov(\Mwf)\geq\lambda$ in $V^4$, but this is because, for each regular $\nu$ such that
   $\mu_3<\nu<\lambda$, in the $\nu$-middle generic stage $\nu\leq\cov(\Mwf)$ holds and it is preserved until the final extension of the iteration.
\end{proof}
\begin{theorem}\label{LeftCichonDomNonNLarge}
   In $V$, assume $\cf(\lambda)\geq\mu_3$. Then, there exists a ccc poset that forces $\mathbf{GCH}_\lambda$,
   $\add(\Nwf)=\mu_1$, $\cov(\Nwf)=\mu_2$, $\bfrak=\mu_3$, $\non(\Mwf)=\cov(\Mwf)=\kappa$ and $\dfrak=\non(\Nwf)=\cfrak=\lambda$.
\end{theorem}
\begin{proof}
  Continue where we left in the proof of Theorem \ref{LeftCichonCovMLarge}. Note that $(\blacktriangle,<^*,C,\mu_3)$ holds
  because it holds in $V^3$ and it is preserved in $V^4$, as the iteration that generates this extension satisfies $(+^{\mu_3}_{\cdot,<^*})$.
  Now, perform a fsi
   $\langle\Por_\alpha,\Qnm_\alpha\rangle_{\alpha<\kappa}$ such that
   \begin{itemize}
      \item for $\alpha\equiv0\modulo4$, $\Qnm_\alpha=\Enm$ ($\Por_\alpha$-name for $\Eor$),
      \item for $\alpha\equiv1\modulo4$, $\Qnm_\alpha$ is the $\Por_\alpha$-name for the fsp (finite support product) of size $\lambda$
            of \emph{all} the subalgebras of $\Aor$ of size $<\mu_1$ in any $\Por_\alpha$-generic extension of $V^4$,
      \item for $\alpha\equiv2\modulo4$, $\Qnm_\alpha$ is the $\Por_\alpha$-name for fsp of size $\lambda$
            of \emph{all} the subalgebras of $\Bor$ of size $<\mu_2$ in any $\Por_\alpha$-generic extension of $V^4$, and
      \item for $\alpha\equiv3\modulo4$, $\Qnm_\alpha$ is the $\Por_\alpha$-name for fsp of size $\lambda$
            of \emph{all} the subalgebras of $\Dor$ of size $<\mu_2$ in any $\Por_\alpha$-generic extension of $V^4$.
   \end{itemize}
   Note that $\Qnm_\alpha$ in the last three cases can be defined as a fsi of length $\lambda$ of subalgebras of small size. It is easy to note
   that the iteration satisfies $(+^{\mu_1}_{\cdot,\subseteq^*})$, $(+^{\mu_2}_{\cdot,\pitchfork})$ and $(+^{\mu_3}_{\cdot,<^*})$ so,
   in a generic extension $V^5$ of this iteration,
   by the arguments as in the previous theorem, $(\blacktriangle,\subseteq^*,A,\mu_1)$, $(\blacktriangle,\pitchfork,B,\mu_2)$ and
   $(\blacktriangle,<^*,C,\mu_3)$ are preserved, so
   $\add(\Nwf)=\mu_1$, $\cov(\Nwf)=\mu_2$ and $\bfrak=\mu_3$ (the $\geq$'s are obtained likewise using the small subalgebras of the iteration).
   Because of the $\kappa$-many Cohen and eventually different reals added, we get
   $\cov(\Mwf)=\non(\Mwf)=\kappa$. $\dfrak=\non(\Nwf)=\cfrak=\lambda$ because $\dfrak\geq\lambda$, $\non(\Nwf)\geq\lambda$ and
   $\mathbf{GCH}_\lambda$ are preserved in the iteration.
\end{proof}
\begin{theorem}\label{LeftCichonDomLarge}
   In $V$, assume $\cf(\lambda)\geq\mu_2$. Then, there exists a ccc poset that forces $\mathbf{GCH}_\lambda$,
   $\add(\Nwf)=\mu_1$, $\bfrak=\mu_2$, $\cov(\Nwf)=\non(\Mwf)=\cov(\Mwf)=\non(\Nwf)=\kappa$ and $\dfrak=\cfrak=\lambda$.
\end{theorem}
\begin{proof}
   We start with the model $V^4$ of Theorem \ref{LeftCichonCovMLarge} with $\mu_3=\mu_2$. Perform a fsi
   $\langle\Por_\alpha,\Qnm_\alpha\rangle_{\alpha<\kappa}$ such that
   \begin{itemize}
      \item for $\alpha\equiv0\modulo3$, $\Qnm_\alpha=\Bnm$ ($\Por_\alpha$-name for $\Bor$),
      \item for $\alpha\equiv1\modulo3$, $\Qnm_\alpha$ is the $\Por_\alpha$-name for the fsp of size $\lambda$
            of \emph{all} the subalgebras of $\Aor$ of size $<\mu_1$ in any $\Por_\alpha$-generic extension of $V^4$, and
      \item for $\alpha\equiv2\modulo3$, $\Qnm_\alpha$ is the $\Por_\alpha$-name for fsp of size $\lambda$
            of \emph{all} the subalgebras of $\Dor$ of size $<\mu_2$ in any $\Por_\alpha$-generic extension of $V^4$.
   \end{itemize}
\end{proof}
\begin{theorem}\label{LeftCichonNonNLarge}
   In $V$, assume $\cf(\lambda)\geq\mu_2$. Then, there exists a ccc poset that forces $\mathbf{GCH}_\lambda$,
   $\add(\Nwf)=\mu_1$, $\cov(\Nwf)=\mu_2$, $\add(\Mwf)=\cof(\Mwf)=\kappa$ and $\non(\Nwf)=\cfrak=\lambda$.
\end{theorem}
\begin{proof}
   Start with the model $V^4$ of the proof of Theorem \ref{LeftCichonCovMLarge} with $\mu_3=\mu_2$. Perform a fsi
   $\langle\Por_\alpha,\Qnm_\alpha\rangle_{\alpha<\kappa}$ such that
   \begin{itemize}
      \item for $\alpha\equiv0\modulo3$, $\Qnm_\alpha=\Dnm$ ($\Por_\alpha$-name for $\Dor$),
      \item for $\alpha\equiv1\modulo3$, $\Qnm_\alpha$ is the $\Por_\alpha$-name for the fsp of size $\lambda$
            of \emph{all} the subalgebras of $\Aor$ of size $<\mu_1$ in any $\Por_\alpha$-generic extension of $V^4$, and
      \item for $\alpha\equiv2\modulo3$, $\Qnm_\alpha$ is the $\Por_\alpha$-name for fsp of size $\lambda$
            of \emph{all} the subalgebras of $\Bor$ of size $<\mu_2$ in any $\Por_\alpha$-generic extension of $V^4$.
   \end{itemize}
\end{proof}
\begin{theorem}\label{LefCichonCofNLarge}
   In $V$, assume $\cf(\lambda)\geq\mu_1$. Then, there exists a ccc poset that forces $\mathbf{GCH}_\lambda$,
   $\add(\Nwf)=\mu_1$, $\cov(\Nwf)=\add(\Mwf)=\cof(\Mwf)=\non(\Nwf)=\kappa$ and $\cof(\Nwf)=\cfrak=\lambda$.
\end{theorem}
\begin{proof}
   Start with the model $V^4$ of the proof of Theorem \ref{LeftCichonCovMLarge} with $\mu_2=\mu_3=\mu_1$. Perform a fsi
   $\langle\Por_\alpha,\Qnm_\alpha\rangle_{\alpha<\kappa}$ such that
   \begin{itemize}
      \item for $\alpha\equiv0\modulo3$, $\Qnm_\alpha=\Dnm$ ($\Por_\alpha$-name for $\Dor$),
      \item for $\alpha\equiv1\modulo3$, $\Qnm_\alpha=\Bnm$ ($\Por_\alpha$-name for $\Bor$), and
      \item for $\alpha\equiv2\modulo3$, $\Qnm_\alpha$ is the $\Por_\alpha$-name for the fsp of size $\lambda$
            of \emph{all} the subalgebras of $\Aor$ of size $<\mu_1$ in any $\Por_\alpha$-generic extension of $V^4$.
   \end{itemize}
\end{proof}

\section{Preservation of $\sqsubset$-unbounded reals}\label{SecPresUnbdg}

Notions and results of this section are fundamental to the construction of matrix iterations in section \ref{SecMatrixit} and for the
construction of the models for our main results in section \ref{SecModelRight}.\\
Recall Context \ref{ContextUnbd}. Throughout this section, fix $M\subseteq N$ models of $\thzfc$ and $c\in\omega^\omega\cap N$ a
$\sqsubset$-unbounded real over $M$.
\begin{definition}\label{DefCompSubordM}
   Given $\Por\in M$ and $\Qor$ posets, we say that \emph{$\Por$ is a complete suborder of $\Qor$ with respect to $M$}, denoted by
   $\Por\preceq_M\Qor$, if $\Por\subseteq\Qor$ and all maximal antichains of $\Por$ in $M$ are maximal antichains of $\Qor$.
\end{definition}
The main consequence of this definition is that, whenever $\Por\in M$ and $\Qor\in N$ are posets such that $\Por\preceq_M\Qor$ then, whenever
$G$ is $\Qor$-generic over $N$, $\Por\cap G$ is a $\Por$-generic set over $M$. Here, we are interested in the case where the real $c$ can be
preserved to be $\sqsubset$-unbounded over $M[G\cap\Por]$.
\begin{definition}\label{DefPresUnbdg}
   Assume $\Por\in M$ and $\Qor\in N$ posets such that $\Por\preceq_M\Qor$. We say that the property $(\star,\Por,\Qor,M,N,\sqsubset,c)$
   holds iff, for every $\dot{h}\in M$ $\Por$-name for a real in $\omega^\omega$, $\Vdash_{\Qor,N}c\not\sqsubset\dot{h}$.
   This is equivalent to saying that $\Vdash_{\Qor,N}$\lcom$c$ is $\sqsubset$-unbounded over $M^{\Por}$\rcom, that is,
   $c$ is $\sqsubset$-unbounded over $M[G\cap\Por]$ for every $G$ $\Qor$-generic over $N$.
\end{definition}
The last two definitions are important notions introduced in \cite{blsh} and \cite{BF} for the preservation of unbounded reals and the
construction of matrix iterations.\\
In relation with the preservation property of Definition \ref{DefPlusProp}, we have the following.
\begin{theorem}\label{suslinStar}
   Let $\Por$ be a Suslin ccc forcing notion with parameters in $M$. If $(+_{\Por,\sqsubset})$ holds in $M$,
   then $(\star,\Por^M,\Por^N,M,N,\sqsubset,c)$ holds.
\end{theorem}
\begin{proof}
It is clear that $\Por^M\preceq_M\Por^N$.
\begin{clm}\label{Pi11Unbdg}
   Let $\dot{h}\in M$ be a $\Por$-name for a real in $\omega^\omega$ and $Y\in M$
   be a countable set of reals in $\omega^\omega$. Then, the statement
   \[\forall_{f\in\omega^\omega}[(\forall_{g\in Y}(f\not\sqsubset g))\imp\Vdash f\not\sqsubset\dot{h}]\]
   is absolute for $M$.
\end{clm}
\begin{proof}
   Work in $M$: $\dot{h}$ can be given by a sequence $\langle W_n\rangle_{n<\omega}$ of maximal antichains in $\Por$, say
   $W_n=\cj{p_{n,m}}{m<\omega}$, and by a sequence $\langle h_n\rangle_{n<\omega}$ of reals in $(\omega^{<\omega})^\omega$ such that
   $p_{n,m}\Vdash$\lcom$\dot{h}\frestr n=h_n(m)$\rcom\ for all $n,m<\omega$. Also, put $Y=\cj{g_k}{k<\omega}$. Note that $\Vdash f\not\sqsubset\dot{h}$ is
   equivalent to the statement
   \[\forall_{p\in\Por}\forall_{k<\omega}\exists_{n,m<\omega}(p\parallel p_{n,m}\conj [h_n(m)]\cap(\sqsubset_k)_f=\vacio),\]
   where $(\sqsubset_k)_f:=\cj{g\in\omega^\omega}{f\sqsubset_k g}$. Therefore, the statement in this Claim is given by a $\Pi_1^1$ formula, so it is absolute.
\end{proof}
   Now, let $\dot{h}\in M$ be a $\Por$-name for a real in $\omega^\omega$. Work in $M$: let $Y\in M$ be a witness of $(+_{\Por,\sqsubset})$ for
   $\dot{h}$. Now, by Claim \ref{Pi11Unbdg}, in $N$ holds
   \[\forall_{f\in\omega^\omega}[(\forall_{g\in Y}(f\not\sqsubset g))\imp\Vdash f\not\sqsubset\dot{h}].\]
   As $c$ is $\sqsubset$-unbounded over $M$, clearly $\forall_{g\in Y}(c\not\sqsubset g)$, so $\Vdash_{\Por,N}c\not\sqsubset\dot{h}$.
\end{proof}
\begin{lemma}[Brendle and Fischer, {\cite[Lemma 11]{BF}}]\label{fixedStar}
   Let $\Por\in M$ be a forcing notion. If $c\in N$ is a $\sqsubset$-unbounded real over $M$, then $(\star,\Por,\Por,M,N,\sqsubset,c)$ holds.
\end{lemma}
\begin{lemma}[{\cite[Lemmas 10 and 13]{BF}}]\label{CompSubordIt}
   Let $\delta$ be an ordinal in $M$, $\Por_{0,\delta}=\langle\Por_{0,\alpha},\Qnm_{0,\alpha}\rangle_{\alpha<\delta}$
     a fsi of posets defined in $M$ and $\Por_{1,\delta}=\langle\Por_{1,\alpha},\Qnm_{1,\alpha}\rangle_{\alpha<\delta}$
     a fsi of posets defined in $N$. Then, $\Por_{0,\delta}\preceq_M\Por_{1,\delta}$ iff, for every $\alpha<\delta$,
     $\Vdash_{\Por_{1,\alpha},N}\Qnm_{0,\alpha}\preceq_{M^{\Por_{0,\alpha}}}\Qnm_{1,\alpha}$.
\end{lemma}
\begin{theorem}[Blass and Shelah, \cite{blsh}, {\cite[Lemma 12]{BF}}]\label{PreservStar}
   With the notation in Lemma \ref{CompSubordIt}, assume that $\Por_{0,\delta}\preceq_M\Por_{1,\delta}$. Then,
   $(\star,\Por_{0,\delta},\Por_{1,\delta},M,N,\sqsubset,c)$ holds iff, for every $\alpha<\delta$,
   \[\Vdash_{\Por_{1,\alpha},N}(\star,\Qnm_{0,\alpha},\Qnm_{1,\alpha},M^{\Por_{0,\alpha}},N^{\Por_{1,\alpha}},\sqsubset,c).\]
\end{theorem}

\section{Matrix iterations of ccc posets}\label{SecMatrixit}

Throughout this section, we work in a model $V$ of $\thzfc$. Fix two ordinals $\delta$ and $\gamma$.
\begin{definition}[Blass and Shelah, \cite{blsh} and \cite{BF}]\label{DefMatrixIt}
A \emph{matrix iteration of ccc posets} is given by $\Por_{\delta,\gamma}=\langle\langle\Por_{\alpha,\xi},\Qnm_{\alpha,\xi}
\rangle_{\xi<\gamma}\rangle_{\alpha\leq\delta}$ with the following conditions.
\begin{enumerate}[(1)]
   \item $\Por_{\delta,0}=\langle\Por_{\alpha,0},\Rnm_{\alpha}\rangle_{\alpha<\delta}$ is a fsi of ccc posets.
   \item For all $\alpha\leq\delta$, $\langle\Por_{\alpha,\xi},\Qnm_{\alpha,\xi}
         \rangle_{\xi<\gamma}$ is a fsi of ccc posets
   \item For all $\xi<\gamma$ and $\alpha<\beta\leq\delta$,
         $\Vdash_{\Por_{\beta,\xi}}\Qnm_{\alpha,\xi}\preceq_{V^{\Por_{\alpha,\xi}}}\Qnm_{\beta,\xi}$.
\end{enumerate}
By Lemma \ref{CompSubordIt}, condition (3) is equivalent to saying that $\Por_{\alpha,\gamma}$ is a complete suborder of $\Por_{\beta,\gamma}$
for
every $\alpha<\beta\leq\delta$.\\
In the context of matrix iterations, when $\alpha\leq\delta$, $\xi\leq\gamma$ and $G_{\alpha,\xi}$ is $\Por_{\alpha,\xi}$-generic over $V$, we
denote
$V_{\alpha,\xi}=V[G_{\alpha,\xi}]$. Note that $V_{0,0}=V$.\\
Figure \ref{fig:2} shows the form in which we think of a matrix iteration. The iteration defined in (1) is represented by the leftmost vertical
iteration and, at each $\alpha$-stage of this iteration ($\alpha\leq\delta$), a horizontal iteration is performed as it is represented in (2).
\begin{figure}
\begin{center}
  \includegraphics[scale=0.48]{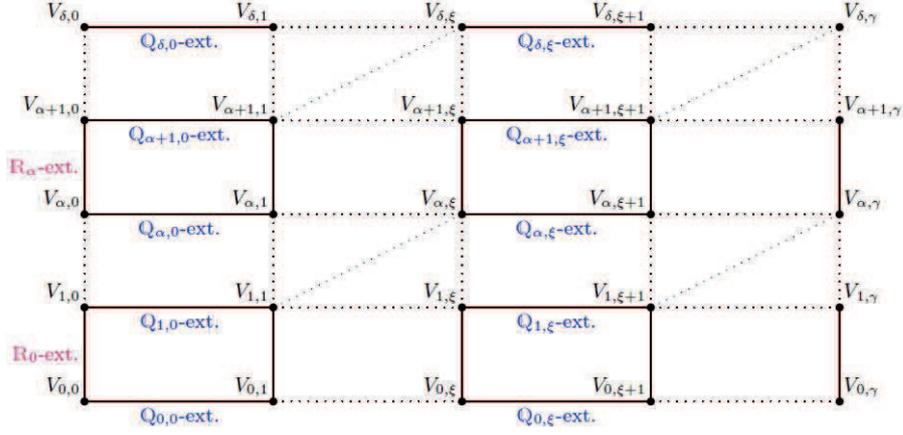}
\caption{Matrix iteration} \label{fig:2}
\end{center}
\end{figure}
\end{definition}
The construction of the matrix iterations for the models in Section \ref{SecModelRight} corresponds to the following particular case, which we
fix from now on.
\begin{context}\label{ContextMatrix}
   Fix a function $\Delta:\gamma\to\delta$ and, for $\xi<\gamma$, let $\Sor_\xi$ be a Suslin ccc poset with parameters in $V$.
   Define the matrix iteration $\Por_{\delta,\gamma}=\langle\langle\Por_{\alpha,\xi},\Qnm_{\alpha,\xi}
   \rangle_{\xi<\gamma}\rangle_{\alpha\leq\delta}$ as follows.
   \begin{enumerate}[(1)]
       \item $\Por_{\delta,0}=\langle\Por_{\alpha,0},\Cnm\rangle_{\alpha<\delta}$ (fsi of Cohen forcing).
       \item For a fixed $\xi<\gamma$, $\Qnm_{\alpha,\xi}$ is defined for all $\alpha\leq\delta$ according to one of the following cases.
   \begin{enumerate}[(i)]
       \item $\Qnm_{\alpha,\xi}=\Snm_\xi$ (as a $\Por_{\alpha,\xi}$-name), or
       \item for a fixed $\Por_{\Delta(\xi),\xi}$-name $\Tnm_\xi$ of a subalgebra of $\Sor_\xi$,
           \[\Qnm_{\alpha,\xi}=\left\{
               \begin{array}{ll}
                   \mathds{1} & \textrm{if $\alpha\leq \Delta(\xi)$,}\\
                   \Tnm_\xi & \textrm{if $\alpha>\Delta(\xi)$.}
               \end{array}
               \right.\]
   \end{enumerate}
   \end{enumerate}
   It is clear that this satisfies the conditions of the Definition \ref{DefMatrixIt}.\\
   From the iteration in (1), for $\alpha<\delta$ let $\dot{c}_\alpha$ be a $\Por_{\alpha+1,0}$-name
   for a Cohen real over $V_{\alpha,0}$. Therefore, from Context \ref{ContextUnbd} it is clear that $\dot{c}_\alpha$ represents a
   $\sqsubset$-unbounded real over $V_{\alpha,0}$ (actually, this is the only place where we use in this paper the second condition of
   Context \ref{ContextUnbd}).
\end{context}
The same argument as in the proof of \cite[Lemma 15]{BF} yields the following.
\begin{theorem}[Brendle and Fischer]\label{realsmatrixbelow}
   Assume that $\delta$ has uncountable cofinality and $\xi\leq\gamma$.
   \begin{enumerate}[(a)]
      \item If $p\in\Por_{\delta,\xi}$ then there exists an $\alpha<\delta$ such that $p\in\Por_{\alpha,\xi}$.
      \item If $\dot{h}$ is a $\Por_{\delta,\xi}$-name for a real, then there exists an $\alpha<\delta$ such that
            $\dot{h}$ is a $\Por_{\alpha,\xi}$-name.
   \end{enumerate}
\end{theorem}
When we go through generic extensions of the matrix iteration, for every $\alpha<\delta$ we are interested in preserving the
$\sqsubset$-unboundedness of $\dot{c}_\alpha$ through the horizontal iterations. The following results state conditions that guarantee this.
\begin{theorem}\label{PresvUnbdg}
   Assume that, for every $\xi<\gamma$ such that all $\Qnm_{\alpha,\xi}$ for $\alpha\leq\delta$ are defined as (2)(i) of Context \ref{ContextMatrix},
   $\Vdash_{\Por_{\alpha,\xi}}(+_{\Qnm_{\alpha,\xi},\sqsubset})$. Then, for all $\alpha<\delta$, $\Por_{\alpha+1,\gamma}$ forces that $\dot{c}_\alpha$
   is a $\sqsubset$-unbounded real over $V_{\alpha,\gamma}$.
\end{theorem}
\begin{proof}
   Fixing $\alpha\leq\delta$, this is easily proved by induction on $\xi\leq\gamma$ by using Lemma \ref{fixedStar} and Theorems \ref{suslinStar}
   and $\ref{PreservStar}$ with $M=V_{\alpha,0}$ and $N=V_{\alpha+1,0}$.
\end{proof}
\begin{corollary}\label{MainMatrixConsq}
   With the same assumptions as in Theorem \ref{PresvUnbdg}, if $\delta$ has uncountable cofinality then
   $\Vdash_{\Por_{\delta,\gamma}}\cf(\delta)\leq\dfrak_\sqsubset$.
\end{corollary}
\begin{proof}
   Let $G$ be $\Por_{\delta,\gamma}$-generic over $V$. Work in $V_{\delta,\gamma}$. Let $F\subseteq\omega^\omega$ with
   $|F|<\cf(\delta)$. By Theorem \ref{realsmatrixbelow}, there exists an $\alpha<\delta$ such that $F\subseteq V_{\alpha,\gamma}$ and,
   as $c_\alpha:=\dot{c}_{\alpha}[G]$ is $\sqsubset$-unbounded over $V_{\alpha,\gamma}$, no real in $F$ can dominate $c_\alpha$.
\end{proof}
By Lemma \ref{InvforPitchfork}, Corollary \ref{MainMatrixConsq} holds for $\sqsubset=\pitchfork$ with $\non(\Nwf)$ in place of
$\dfrak_{\pitchfork}$.

\section{Models for the right hand side of Cichon's diagram}\label{SecModelRight}

Throughout this section our results are given for a model $V$ of $\thzfc$. There, we fix the following. $\mu_1\leq\mu_2\leq\nu\leq\kappa$
uncountable regular cardinals and $\lambda\geq\kappa$ a cardinal, $t:\kappa\nu\to\kappa$ such that $t(\kappa\delta+\alpha)=\alpha$ for
$\delta<\nu$ and $\alpha<\kappa$. The product $\kappa\nu$, as all the products we are going to consider throughout this section, denotes ordinal
product. Also, fix a bijection $g:\lambda\to\kappa\times\lambda$ and
put $(\cdot)_0:\kappa\times\lambda\to\kappa$ the projection onto the first coordinate.\\
All the matrix iterations are defined under the considerations of Context \ref{ContextMatrix}.

\subsection{Models with $\cfrak$ large}\label{SubsecContLarge}

In this part, assume in $V$ that $\cf(\lambda)\geq\aleph_1$ and that the conclusions of Theorem \ref{LeftCichonCovMLarge} hold with
$\mu_1=\mu_2=\mu_3=\aleph_1$ (we can get this model by just adding $\lambda$-many Cohen reals to a model of $\mathbf{GCH}$).
\begin{theorem}\label{RightContLarge1}
   There exists a ccc poset that forces $\mathbf{GCH}_\lambda$, $\add(\Nwf)=\non(\Mwf)$ $=\nu$ and $\cov(\Mwf)=\cof(\Nwf)=\kappa$.
\end{theorem}
\begin{proof}
   Perform a matrix iteration $\Por_{\kappa,\kappa\nu}=
   \langle\langle\Por_{\alpha,\xi},\Qnm_{\alpha,\xi}\rangle_{\xi<\kappa\nu}\rangle_{\alpha\leq\kappa}$ such that, for a fixed $\xi<\kappa\nu$,
   \begin{itemize}
      \item $\Anm_\xi$ is a $\Por_{t(\xi),\xi}$-name for $\Aor^{V_{t(\xi),\xi}}$ and
            \[\Qnm_{\alpha,\xi}=\left\{
               \begin{array}{ll}
                   \mathds{1} & \textrm{if $\alpha\leq t(\xi)$,}\\
                   \Anm_\xi & \textrm{if $\alpha>t(\xi)$.}
               \end{array}
               \right.\]
   \end{itemize}
   Work in an extension $V_{\kappa,\kappa\nu}$ of the matrix iteration. As the hypothesis of Theorem \ref{PresvUnbdg} holds for
   $\sqsubset=\eqcirc$, from Corollary \ref{MainMatrixConsq} we get $\cov(\Mwf)\geq\kappa$. It is clear that $\mathbf{GCH}_\lambda$
   holds and that $\non(\Mwf)\leq\nu$ because of the $\nu$-cofinally many Cohen reals added in the iteration, so it
   remains to prove that $\nu\leq\add(\Nwf)$ and $\cof(\Nwf)\leq\kappa$. Note that, for each
   $\xi<\kappa\nu$, $\Aor_\xi:=\Anm_{t(\xi),\xi}(G)$ adds a Borel null set $N_\xi$ coded in $V_{t(\xi)+1,\xi+1}$ that covers all the
   Borel null sets coded in $V_{t(\xi),\xi}$. To conclude this proof, it is enough to prove the following.
   \begin{clm}\label{ClaimAmoeba}
      Every family of $<\nu$ many Borel null sets coded in $V_{\kappa,\kappa\nu}$ is covered by some $N_\xi$.
   \end{clm}
   \begin{proof}
      Let $\Bwf$ be such a family. By Theorem \ref{realsmatrixbelow}, there exist $\alpha<\kappa$ and $\eta<\kappa\nu$ such that
      all the members of $\Bwf$ are coded in $V_{\alpha,\eta}$. By the definition of $t$, there exists a $\xi\in(\eta,\kappa\nu)$
      such that $t(\xi)=\alpha$, so $N_\xi$ covers all the members of $\Bwf$.
   \end{proof}
\end{proof}
\begin{theorem}\label{RightContLarge2}
   There exists a ccc poset that forces $\mathbf{GCH}_\lambda$, $\add(\Nwf)=\non(\Mwf)$ $=\cov(\Mwf)=\nu$ and
   $\dfrak=\non(\Nwf)=\cof(\Nwf)=\kappa$.
\end{theorem}
\begin{proof}
   Perform a matrix iteration $\Por_{\kappa,\kappa\nu}=
   \langle\langle\Por_{\alpha,\xi},\Qnm_{\alpha,\xi}\rangle_{\xi<\kappa\nu}\rangle_{\alpha\leq\kappa}$ such that, for a fixed $\xi<\kappa\nu$,
   \begin{itemize}
      \item if $\xi\equiv0\modulo2$, $\Qnm_{\alpha,\xi}$ is a $\Por_{\alpha,\xi}$-name of $\Eor^{V_{\alpha,\xi}}$, and
      \item if $\xi\equiv1\modulo2$, $\Anm_\xi$ is a $\Por_{t(\xi),\xi}$-name for $\Aor^{V_{t(\xi),\xi}}$ and
            \[\Qnm_{\alpha,\xi}=\left\{
               \begin{array}{ll}
                   \mathds{1} & \textrm{if $\alpha\leq t(\xi)$,}\\
                   \Anm_\xi & \textrm{if $\alpha>t(\xi)$.}
               \end{array}
               \right.\]
   \end{itemize}
   In an extension $V_{\kappa,\kappa\nu}$, Lemma \ref{EvDiffPlus} and Corollary \ref{MainMatrixConsq} imply that
   $\kappa\leq\dfrak,\non(\Nwf)$ and the $\nu$-cofinally many Cohen and eventually different reals in the iteration
   give $\non(\Mwf),\cov(\Mwf)\leq\nu$. Claim \ref{ClaimAmoeba} also holds in this case (but we get the $N_\xi$ only when $\xi$ is odd),
   so $\nu\leq\add(\Nwf)$ and $\cof(\Nwf)\leq\kappa$.
\end{proof}
The remaining results in this part are proved in a similar fashion as the two previous results.
\begin{theorem}\label{RightContLarge3}
   There exists a ccc poset that forces $\mathbf{GCH}_\lambda$, $\add(\Nwf)=\non(\Mwf)$ $=\non(\Nwf)=\nu$ and
   $\dfrak=\cof(\Nwf)=\kappa$.
\end{theorem}
\begin{proof}
   Perform a matrix iteration $\Por_{\kappa,\kappa\nu}=
   \langle\langle\Por_{\alpha,\xi},\Qnm_{\alpha,\xi}\rangle_{\xi<\kappa\nu}\rangle_{\alpha\leq\kappa}$ such that, for a fixed $\xi<\kappa\nu$,
   \begin{itemize}
      \item if $\xi\equiv0\modulo2$, $\Qnm_{\alpha,\xi}$ is a $\Por_{\alpha,\xi}$-name of $\Bor^{V_{\alpha,\xi}}$, and
      \item if $\xi\equiv1\modulo2$, $\Anm_\xi$ is a $\Por_{t(\xi),\xi}$-name for $\Aor^{V_{t(\xi),\xi}}$ and
            \[\Qnm_{\alpha,\xi}=\left\{
               \begin{array}{ll}
                   \mathds{1} & \textrm{if $\alpha\leq t(\xi)$,}\\
                   \Anm_\xi & \textrm{if $\alpha>t(\xi)$.}
               \end{array}
               \right.\]
   \end{itemize}
\end{proof}
\begin{theorem}\label{RightContLarge4}
   There exists a ccc poset that forces $\mathbf{GCH}_\lambda$, $\add(\Nwf)=\cof(\Mwf)$ $=\nu$ and
   $\non(\Nwf)=\cof(\Nwf)=\kappa$.
\end{theorem}
\begin{proof}
   Perform a matrix iteration $\Por_{\kappa,\kappa\nu}=
   \langle\langle\Por_{\alpha,\xi},\Qnm_{\alpha,\xi}\rangle_{\xi<\kappa\nu}\rangle_{\alpha\leq\kappa}$ such that, for a fixed $\xi<\kappa\nu$,
   \begin{itemize}
      \item if $\xi\equiv0\modulo2$, $\Qnm_{\alpha,\xi}$ is a $\Por_{\alpha,\xi}$-name of $\Dor^{V_{\alpha,\xi}}$, and
      \item if $\xi\equiv1\modulo2$, $\Anm_\xi$ is a $\Por_{t(\xi),\xi}$-name for $\Aor^{V_{t(\xi),\xi}}$ and
            \[\Qnm_{\alpha,\xi}=\left\{
               \begin{array}{ll}
                   \mathds{1} & \textrm{if $\alpha\leq t(\xi)$,}\\
                   \Anm_\xi & \textrm{if $\alpha>t(\xi)$.}
               \end{array}
               \right.\]
   \end{itemize}
\end{proof}
\begin{theorem}\label{RightContLarge5}
   There exists a ccc poset that forces $\mathbf{GCH}_\lambda$, $\add(\Nwf)=\cof(\Mwf)$ $=\non(\Nwf)=\nu$ and
   $\cof(\Nwf)=\kappa$.
\end{theorem}
\begin{proof}
   Perform a matrix iteration $\Por_{\kappa,\kappa\nu}=
   \langle\langle\Por_{\alpha,\xi},\Qnm_{\alpha,\xi}\rangle_{\xi<\kappa\nu}\rangle_{\alpha\leq\kappa}$ such that, for a fixed $\xi<\kappa\nu$,
   \begin{itemize}
      \item if $\xi\equiv0\modulo3$, $\Qnm_{\alpha,\xi}$ is a $\Por_{\alpha,\xi}$-name of $\Dor^{V_{\alpha,\xi}}$,
      \item if $\xi\equiv1\modulo3$, $\Qnm_{\alpha,\xi}$ is a $\Por_{\alpha,\xi}$-name of $\Bor^{V_{\alpha,\xi}}$, and
      \item if $\xi\equiv2\modulo3$, $\Anm_\xi$ is a $\Por_{t(\xi),\xi}$-name for $\Aor^{V_{t(\xi),\xi}}$ and
            \[\Qnm_{\alpha,\xi}=\left\{
               \begin{array}{ll}
                   \mathds{1} & \textrm{if $\alpha\leq t(\xi)$,}\\
                   \Anm_\xi & \textrm{if $\alpha>t(\xi)$.}
               \end{array}
               \right.\]
   \end{itemize}
\end{proof}

\subsection{Models with $\cof(\Nwf)$ large}\label{SubsecCofNLarge}

Assume in $V$ that $\cf(\lambda)\geq\mu_1$ and that the conclusions of Theorem \ref{LeftCichonCovMLarge} hold with $\mu_2=\mu_3=\mu_1$.
\begin{theorem}\label{RightCofNLarge1}
   There exists a ccc poset that forces $\mathbf{GCH}_\lambda$, $\add(\Nwf)=\mu_1$, $\cov(\Nwf)=\bfrak=\non(\Mwf)=\nu$,
   $\cov(\Mwf)=\dfrak=\non(\Nwf)=\kappa$ and $\cof(\Nwf)=\cfrak=\lambda$.
\end{theorem}
\begin{proof}
  Perform a matrix iteration $\Por_{\kappa,\lambda\kappa\nu}=
  \langle\langle\Por_{\alpha,\xi},\Qnm_{\alpha,\xi}\rangle_{\xi<\lambda\kappa\nu}\rangle_{\alpha\leq\kappa}$ according to the following cases
  for $\rho<\kappa\nu$.
  \begin{enumerate}[(i)]
     \item If $\xi=\lambda\rho$, $\Bnm_\xi$ is a $\Por_{t(\rho),\xi}$-name for $\Bor^{V_{t(\rho),\xi}}$ and
            \[\Qnm_{\alpha,\xi}=\left\{
               \begin{array}{ll}
                   \mathds{1} & \textrm{if $\alpha\leq t(\rho)$,}\\
                   \Bnm_\xi & \textrm{if $\alpha>t(\rho)$.}
               \end{array}
               \right.\]
     \item If $\xi=\lambda\rho+1$, $\Dnm_\xi$ is a $\Por_{t(\rho),\xi}$-name for $\Dor^{V_{t(\rho),\xi}}$ and
            \[\Qnm_{\alpha,\xi}=\left\{
               \begin{array}{ll}
                   \mathds{1} & \textrm{if $\alpha\leq t(\rho)$,}\\
                   \Dnm_\xi & \textrm{if $\alpha>t(\rho)$.}
               \end{array}
               \right.\]
  \end{enumerate}
  To conclude the construction of the matrix iteration, fix, for each $\alpha\leq\kappa$, a sequence
  $\langle\Anm_{\alpha,\gamma}^\rho\rangle_{\gamma<\lambda}$ of $\Por_{\alpha,\lambda\rho+2}$-names for \emph{all}
  suborders of $\Aor^{V_{\alpha,\lambda\rho+2}}$ of size $<\mu_1$.
  \begin{enumerate}[(i)]
  \setcounter{enumi}{2}
     \item If $\xi=\lambda\rho+2+\epsilon$ ($\epsilon<\lambda$), put
           \[\Qnm_{\alpha,\xi}=\left\{
                            \begin{array}{ll}
                               \mathds{1}, & \textrm{if $\alpha\leq(g(\epsilon))_0$,}\\
                               \Anm^\rho_{g(\epsilon)}, & \textrm{if $\alpha>(g(\epsilon))_0$.}
                            \end{array}
                     \right.\]
  \end{enumerate}
  By Lemmas \ref{smallPlus}, \ref{centeredAddNull} and \ref{spamPlus} and Theorem \ref{preservPlus}, the matrix iteration satisfies
  $(+_{\cdot,\subseteq^*}^{\mu_1})$.
  Work in an extension $V_{\kappa,\lambda\kappa\nu}$ of the matrix iteration. From the part (iii) of the construction, using the same argument
  as in the proofs of section \ref{SecModelLeft} and Theorem \ref{realsmatrixbelow}, $\add(\Nwf)=\mu_1$. $\cof(\Nwf)\geq\lambda$ is given by Lemma \ref{PreserTriangle} and \lcom$(+_{\cdot,\subseteq^*}^{\mu_1})$
  and $\cof(\Nwf)=\lambda$\rcom\ in the ground model. $\cfrak\leq\lambda$ is clear. Because
  of the $\nu$-cofinally many Cohen reals, $\non(\Mwf)\leq\nu$ and, by Corollary \ref{MainMatrixConsq} with $\sqsubset=\eqcirc$,
  $\cov(\Mwf)\geq\kappa$.\\
  We need to get $\dfrak,\non(\Nwf)\leq\kappa$ and $\bfrak,\cov(\Nwf)\geq\nu$. Note that, for each
  $\rho<\kappa\nu$, $\Bor_\rho:=\Bnm_{t(\rho),\lambda\rho}(G)$ adds a random real $b_\rho\in V_{t(\rho)+1,\lambda\rho+1}$
  over $V_{t(\rho),\lambda\rho}$ and $\Dor_\rho:=\Dnm_{t(\rho),\lambda\rho+1}(G)$ adds a dominating real
  $d_\rho\in V_{t(\rho)+1,\lambda\rho+2}$ over $V_{t(\rho),\lambda\rho+1}$.
  To finish the proof, it is enough to prove the two results ahead.
  \begin{clm}\label{claimRandom}
     For every family of Borel null sets in $V_{\kappa,\lambda\kappa\nu}$ of size $<\nu$, there is a $b_\rho$ that is not in
     its union.
  \end{clm}
  \begin{proof}
     Let $\Bwf$ be such a family. Theorem \ref{realsmatrixbelow} implies that all the sets in $\Bwf$ are coded in $V_{\alpha,\lambda\delta}$ for
     some $\alpha<\kappa$ and $\delta<\kappa\nu$. Then, there exists a $\rho\in(\delta,\kappa\nu)$ such that
     $t(\rho)=\alpha$. Then, $b_\rho$ is such a real.
  \end{proof}
  With a similar argument, we can prove
  \begin{clm}\label{claimDom}
     Every family of reals in $V_{\kappa,\lambda\kappa\nu}$ of size $<\nu$ is dominated by some $d_\rho$.
  \end{clm}
\end{proof}
\begin{theorem}\label{RightCofNLarge2}
   There exists a ccc poset that forces $\mathbf{GCH}_\lambda$, $\add(\Nwf)=\mu_1$, $\cov(\Nwf)=\bfrak=\non(\Mwf)=\cov(\Mwf)=\nu$,
   $\dfrak=\non(\Nwf)=\kappa$ and $\cof(\Nwf)=\cfrak=\lambda$.
\end{theorem}
\begin{proof}
  Perform a matrix iteration $\Por_{\kappa,\lambda\kappa\nu}=
  \langle\langle\Por_{\alpha,\xi},\Qnm_{\alpha,\xi}\rangle_{\xi<\lambda\kappa\nu}\rangle_{\alpha\leq\kappa}$ according to the following cases
  for $\rho<\kappa\nu$.
  \begin{enumerate}[(i)]
     \item If $\xi=\lambda\rho$, $\Qnm_{\alpha,\xi}$ is a $\Por_{\alpha,\xi}$-name for $\Eor$.
     \item If $\xi=\lambda\rho+1$, $\Bnm_\xi$ is a $\Por_{t(\rho),\xi}$-name for $\Bor^{V_{t(\rho),\xi}}$ and
            \[\Qnm_{\alpha,\xi}=\left\{
               \begin{array}{ll}
                   \mathds{1} & \textrm{if $\alpha\leq t(\rho)$,}\\
                   \Bnm_\xi & \textrm{if $\alpha>t(\rho)$.}
               \end{array}
               \right.\]
     \item If $\xi=\lambda\rho+2$, $\Dnm_\xi$ is a $\Por_{t(\rho),\xi}$-name for $\Dor^{V_{t(\rho),\xi}}$ and
            \[\Qnm_{\alpha,\xi}=\left\{
               \begin{array}{ll}
                   \mathds{1} & \textrm{if $\alpha\leq t(\rho)$,}\\
                   \Dnm_\xi & \textrm{if $\alpha>t(\rho)$.}
               \end{array}
               \right.\]
  \end{enumerate}
  Like in the previous proof, fix, for each $\alpha\leq\kappa$, a sequence
  $\langle\Anm_{\alpha,\gamma}^\rho\rangle_{\gamma<\lambda}$ of $\Por_{\alpha,\lambda\rho+3}$-names for \emph{all}
  suborders of $\Aor^{V_{\alpha,\lambda\rho+3}}$ of size $<\mu_1$.
  \begin{enumerate}[(i)]
  \setcounter{enumi}{3}
     \item If $\xi=\lambda\rho+3+\epsilon$ ($\epsilon<\lambda$), put
           \[\Qnm_{\alpha,\xi}=\left\{
                            \begin{array}{ll}
                               \mathds{1}, & \textrm{if $\alpha\leq(g(\epsilon))_0$,}\\
                               \Anm^\rho_{g(\epsilon)}, & \textrm{if $\alpha>(g(\epsilon))_0$.}
                            \end{array}
                     \right.\]
  \end{enumerate}
  With the same argument as in Theorem \ref{RightCofNLarge1}, in an extension $V_{\kappa,\lambda\kappa\nu}$ we get that
  $\cof(\Nwf)=\cfrak=\lambda$, $\add(\Nwf)=\mu_1$, $\cov(\Nwf),\bfrak\geq\nu$ and $\dfrak=\non(\Nwf)=\kappa$,
  but for the last we consider Corollary \ref{MainMatrixConsq}
  for the cases when $\sqsubset$ is $\pitchfork$ or $<^*$. The $\nu$-cofinally many Cohen and eventually different reals
  yield $\non(\Mwf)=\cov(\Mwf)=\nu$.
\end{proof}
\begin{theorem}\label{RightCofNLarge3}
   There exists a ccc poset that forces $\mathbf{GCH}_\lambda$, $\add(\Nwf)=\mu_1$, $\cov(\Nwf)=\add(\Mwf)=\cof(\Mwf)=\nu$,
   $\non(\Nwf)=\kappa$ and $\cof(\Nwf)=\cfrak=\lambda$.
\end{theorem}
\begin{proof}
  Perform a matrix iteration $\Por_{\kappa,\lambda\kappa\nu}=
  \langle\langle\Por_{\alpha,\xi},\Qnm_{\alpha,\xi}\rangle_{\xi<\lambda\kappa\nu}\rangle_{\alpha\leq\kappa}$ according to the following cases
  for $\rho<\kappa\nu$.
  \begin{enumerate}[(i)]
     \item If $\xi=\lambda\rho$, $\Qnm_{\alpha,\xi}$ is a $\Por_{\alpha,\xi}$-name for $\Dor$.
     \item If $\xi=\lambda\rho+1$, $\Bnm_\xi$ is a $\Por_{t(\rho),\xi}$-name for $\Bor^{V_{t(\rho),\xi}}$ and
            \[\Qnm_{\alpha,\xi}=\left\{
               \begin{array}{ll}
                   \mathds{1} & \textrm{if $\alpha\leq t(\rho)$,}\\
                   \Bnm_\xi & \textrm{if $\alpha>t(\rho)$.}
               \end{array}
               \right.\]
     \item As (iii) in Theorem \ref{RightCofNLarge1}.
  \end{enumerate}
  Use Claim \ref{claimRandom} and the argument of the previous results.
\end{proof}
\begin{theorem}\label{RightCofNLarge4}
   There exists a ccc poset that forces $\mathbf{GCH}_\lambda$, $\add(\Nwf)=\mu_1$, $\cov(\Nwf)=\non(\Mwf)=\add(\Mwf)=\non(\Nwf)=\nu$,
   $\dfrak=\kappa$ and $\cof(\Nwf)=\cfrak=\lambda$.
\end{theorem}
\begin{proof}
  Perform a matrix iteration $\Por_{\kappa,\lambda\kappa\nu}=
  \langle\langle\Por_{\alpha,\xi},\Qnm_{\alpha,\xi}\rangle_{\xi<\lambda\kappa\nu}\rangle_{\alpha\leq\kappa}$ according to the following cases
  for $\rho<\kappa\nu$.
  \begin{enumerate}[(i)]
     \item If $\xi=\lambda\rho$, $\Qnm_{\alpha,\xi}$ is a $\Por_{\alpha,\xi}$-name for $\Bor$.
     \item If $\xi=\lambda\rho+1$, $\Dnm_\xi$ is a $\Por_{t(\rho),\xi}$-name for $\Dor^{V_{t(\rho),\xi}}$ and
            \[\Qnm_{\alpha,\xi}=\left\{
               \begin{array}{ll}
                   \mathds{1} & \textrm{if $\alpha\leq t(\rho)$,}\\
                   \Dnm_\xi & \textrm{if $\alpha>t(\rho)$.}
               \end{array}
               \right.\]
     \item As (iii) in Theorem \ref{RightCofNLarge1}.
  \end{enumerate}
  Use Claim \ref{claimDom}.
\end{proof}

\subsection{Models with $\non(\Nwf)$ large}\label{SubsecNonNLarge}

Assume in $V$ that $\cf(\lambda)\geq\mu_2$ and that the conclusions of Theorem \ref{LeftCichonCovMLarge} hold with $\mu_3=\mu_2$.
\begin{theorem}\label{RightNonNLarge1}
    There exists a ccc poset that forces $\mathbf{GCH}_\lambda$, $\add(\Nwf)=\mu_1$, $\cov(\Nwf)=\mu_2$, $\bfrak=\non(\Mwf)=\nu$,
    $\cov(\Mwf)=\dfrak=\kappa$ and $\non(\Nwf)=\cfrak=\lambda$.
\end{theorem}
\begin{proof}
  Perform a matrix iteration $\Por_{\kappa,\lambda\kappa\nu}=
  \langle\langle\Por_{\alpha,\xi},\Qnm_{\alpha,\xi}\rangle_{\xi<\lambda\kappa\nu}\rangle_{\alpha\leq\kappa}$ according to the following cases
  for $\rho<\kappa\nu$.
  \begin{enumerate}[(i)]
     \item If $\xi=\lambda\rho$, $\Dnm_\xi$ is a $\Por_{t(\rho),\xi}$-name for $\Dor^{V_{t(\rho),\xi}}$ and
            \[\Qnm_{\alpha,\xi}=\left\{
               \begin{array}{ll}
                   \mathds{1} & \textrm{if $\alpha\leq t(\rho)$,}\\
                   \Dnm_\xi & \textrm{if $\alpha>t(\rho)$.}
               \end{array}
               \right.\]
  \end{enumerate}
  Fix, for each $\alpha\leq\kappa$, two sequences
  $\langle\Anm_{\alpha,\gamma}^\rho\rangle_{\gamma<\lambda}$ and $\langle\Bnm_{\alpha,\gamma}^\rho\rangle_{\gamma<\lambda}$ of
  $\Por_{\alpha,\lambda\rho+1}$-names for \emph{all}
  suborders of $\Aor^{V_{\alpha,\lambda\rho+1}}$ of size $<\mu_1$ and \emph{all} suborders of $\Bor^{V_{\alpha,\lambda\rho+1}}$ of size
  $<\mu_2$.
  \begin{enumerate}[(i)]
  \setcounter{enumi}{1}
     \item If $\xi=\lambda\rho+1+2\epsilon$ ($\epsilon<\lambda$), put
           \[\Qnm_{\alpha,\xi}=\left\{
                            \begin{array}{ll}
                               \mathds{1}, & \textrm{if $\alpha\leq(g(\epsilon))_0$,}\\
                               \Anm^\rho_{g(\epsilon)}, & \textrm{if $\alpha>(g(\epsilon))_0$.}
                            \end{array}
                     \right.\]
     \item If $\xi=\lambda\rho+1+2\epsilon+1$ ($\epsilon<\lambda$), put
           \[\Qnm_{\alpha,\xi}=\left\{
                            \begin{array}{ll}
                               \mathds{1}, & \textrm{if $\alpha\leq(g(\epsilon))_0$,}\\
                               \Bnm^\rho_{g(\epsilon)}, & \textrm{if $\alpha>(g(\epsilon))_0$.}
                            \end{array}
                     \right.\]
  \end{enumerate}
  By Lemmas \ref{smallPlus}, \ref{spamPlus}, \ref{centeredFork} and \ref{spamPlus} and Theorem \ref{preservPlus}, the matrix iteration satisfies
  $(+_{\cdot,\subseteq^*}^{\mu_1})$ and $(+_{\cdot,\pitchfork}^{\mu_2})$.
  Work in an extension $V_{\kappa,\lambda\kappa\nu}$ of the matrix iteration. From (ii) and (iii), using the same argument
  as in the proofs in subsection \ref{SubsecCofNLarge}, $\add(\Nwf)=\mu_1$, $\cov(\Nwf)=\mu_2$ and $\non(\Nwf)=\cfrak=\lambda$. Because
  of the $\nu$-cofinally many Cohen reals, $\non(\Mwf)\leq\nu$ and, by Corollary \ref{MainMatrixConsq} with $\sqsubset=\eqcirc$,
  $\cov(\Mwf)\geq\kappa$. Note that Claim \ref{claimDom} holds in this case (with some variations in the subindex of the dominating reals),
  so $\bfrak\geq\nu$ and $\dfrak\leq\kappa$.
\end{proof}
\begin{theorem}\label{RightNonNLarge2}
    There exists a ccc poset that forces $\mathbf{GCH}_\lambda$, $\add(\Nwf)=\mu_1$, $\cov(\Nwf)=\mu_2$, $\bfrak=\non(\Mwf)=\cov(\Mwf)=\nu$,
    $\dfrak=\kappa$ and $\non(\Nwf)=\cfrak=\lambda$.
\end{theorem}
\begin{proof}
  Perform a matrix iteration $\Por_{\kappa,\lambda\kappa\nu}=
  \langle\langle\Por_{\alpha,\xi},\Qnm_{\alpha,\xi}\rangle_{\xi<\lambda\kappa\nu}\rangle_{\alpha\leq\kappa}$ according to the following cases
  for $\rho<\kappa\nu$.
  \begin{enumerate}[(i)]
     \item If $\xi=\lambda\rho$, let $\Qnm_{\alpha,\xi}$ be a $\Por_{\alpha,\xi}$-name for $\Eor$.
     \item If $\xi=\lambda\rho+1$, $\Dnm_\xi$ is a $\Por_{t(\rho),\xi}$-name for $\Dor^{V_{t(\rho),\xi}}$ and
            \[\Qnm_{\alpha,\xi}=\left\{
               \begin{array}{ll}
                   \mathds{1} & \textrm{if $\alpha\leq t(\rho)$,}\\
                   \Dnm_\xi & \textrm{if $\alpha>t(\rho)$.}
               \end{array}
               \right.\]
     \item If $\xi=\lambda\rho+2+2\epsilon$ ($\epsilon<\lambda$), proceed like in (ii) of the proof of Theorem \ref{RightNonNLarge1}.
     \item If $\xi=\lambda\rho+2+2\epsilon+1$ ($\epsilon<\lambda$), proceed like in (iii) of the proof of Theorem \ref{RightNonNLarge1}.
  \end{enumerate}
\end{proof}

\section{Questions}\label{SecQ}

\begin{enumerate}[(1)]
   \item Is there a method that allows us to get a model of a case where more than 3 invariants in the right hand side of Cichon's diagram
         can take arbitrary different values? In particular, can we get a model of $\cov(\Mwf)<\dfrak<\non(\Nwf)<\cof(\Nwf)$?
   \item Can we extend our results to singular cardinals? Specifically, under which conditions can the following cardinals be singular.
         \begin{enumerate}[(a)]
            \item $\mu_2$ in Theorems \ref{LeftCichonCovMLarge}, \ref{LeftCichonDomNonNLarge}, \ref{LeftCichonNonNLarge}, \ref{RightNonNLarge1}
                  and \ref{RightNonNLarge2}.
            \item $\kappa$ in Theorems \ref{LeftCichonDomNonNLarge}, \ref{LeftCichonDomLarge} and in the results of section \ref{SecModelRight}.
         \end{enumerate}
\end{enumerate}

\begin{acknowledgements}
  I would like to thank my PhD advisor, professor J\"org Brendle, for all the constructive discussions about all the material of this paper,
  for pointing out many details that helped me to improve the concepts and results of the preliminary versions of this work and for helping me
  with the corrections of this material. Ideas in section \ref{SecModelLeft} are improvements that professor Brendle made over his own work in
  \cite{brendle} and that he kindly taught me and let me use in this paper.
\end{acknowledgements}

\bibliographystyle{spmpsci}      


\end{document}